\documentclass[12pt,leqno]{amsart}
\usepackage{amsfonts,amsthm,amsmath,comment,xcolor,soul}

\theoremstyle{plain}
\newtheorem{thm}{Theorem}[section]
\newtheorem{prop}{Proposition}[section]
\newtheorem{lem}{Lemma}[section]
\newtheorem{cor}{Corollary}[section]

\theoremstyle{definition}
\newtheorem{df}{Definition}[section]
\newtheorem{rem}{Remark}[section]
\newtheorem{ex}{Example}[section]

\newcommand{\FF}{\mathbb{F}}
\newcommand{\ZZ}{\mathbb{Z}}
\newcommand{\RR}{\mathbb{R}}
\newcommand{\CC}{\mathbb{C}}

\newcommand{\Z}{\mathbb{Z}}

\newcommand{\R}{\mathbb{R}}

\def\bm#1{\mathbf{#1}}

\DeclareMathOperator{\supp}{supp}
\DeclareMathOperator{\Supp}{Supp}

\DeclareMathOperator{\wt}{wt}

\DeclareMathOperator{\Harm}{Harm}

\begin{document}

\title{{
Harmonic Tutte polynomials of matroids II
}
}

\author[Britz]{Thomas Britz}
\address
	{
		School of Mathematics and Statistics,
		University of New South Wales\\
		Sydney, NSW 2052, Australia
	}
\email{britz@unsw.edu.au}

\author[Chakraborty]{Himadri Shekhar Chakraborty*}
\thanks{*Corresponding author}
\address
	{
		Department of Mathematics, 
		Shahjalal University of Science and Technology\\ Sylhet-3114, Bangladesh\\
	}
\email{himadri-mat@sust.edu}

\author[Ishikawa]{Reina Ishikawa}
\address
	{
		Faculty of Science and Engineering, 
		Waseda University, 
		Tokyo 169-8555, Japan\\ 
	}
\email{reina.i@suou.waseda.jp}

\author[Miezaki]{Tsuyoshi Miezaki}
\address
	{
		Faculty of Science and Engineering, 
		Waseda University, 
		Tokyo 169-8555, Japan\\ 
	}
\email{miezaki@waseda.jp}

\author[Tang]{Hopein Christofen Tang}
\address
	{
		School of Mathematics and Statistics,
		University of New South Wales\\
		Sydney, NSW 2052, Australia	
	}
\email{hopein.tang@unsw.edu.au}

\date{}
\maketitle

\begin{abstract}
In this work, we introduce the harmonic generalization of 
the $m$-tuple weight enumerators of codes over finite Frobenius rings.
A harmonic version of the MacWilliams-type identity for $m$-tuple weight enumerators
of codes over finite Frobenius ring is also given.
Moreover, we define the demi-matroid analogue of well-known polynomials from 
matroid theory, namely Tutte polynomials and coboundary polynomials, 
and associate them with a harmonic function. 
We also prove the Greene-type identity relating these polynomials 
to the harmonic $m$-tuple weight enumerators of codes over finite Frobenius rings.
As an application of this Greene-type identity,
we provide a simple combinatorial proof of the MacWilliams-type identity for 
harmonic $m$-tuple weight enumerators over finite Frobenius rings. 
Finally, we provide the structure of the relative invariant spaces
containing the harmonic $m$-tuple weight enumerators of self-dual codes over finite fields.
\end{abstract}

{\small
\vspace{-0.15in}

\quad
}


\section{Introduction}

In 1963, MacWilliams~\cite{MacWilliams} introduced a powerful identity 
relating the weight enumerator of a linear code to that of its dual.
This identity, called the MacWilliams Identity, 
has many important applications in the theory of algebraic coding,
including the central problem of code classification.
In 1976, Greene~\cite{Greene1976} showed that the weight enumerator of a linear code
is obtained by evaluating the Tutte polynomial of the matroid associated with the code.
Using this identity, now known as Greene's Theorem, Greene provided 
a short and elegant new proof of the MacWilliams Identity.
This was the first significant demonstration of how matroid theory relates usefully to coding theory.

The MacWilliams Identity and Greene's Theorem 
have since been generalized by many authors.
Britz and Shiromoto~\cite{BS2008} 
generalized the MacWilliams Identity with respect to matroids.
They used this result to generalize Greene's Theorem with respect to $m$-tuples of codewords
and to provide a short proof of the $m$-tuple generalization of 
the MacWilliams Identity due to Shiromoto~\cite{shiromoto96}.
A different type of generalization of the MacWilliams Identity was 
introduced by Bachoc~\cite{Bachoc} for binary linear codes.
Bachoc generalized the weight enumerator of a linear code with respect to the 
discrete harmonic functions on a finite set introduced by Delsarte~\cite{Delsarte},
and presented a formula for computing harmonic functions using Hahn polynomials~\cite{Delsarte, Karlin}.
Bachoc~\cite{BachocNonBinary} and Tanabe~\cite{Tanabe2001} independently 
extended this harmonic weight MacWilliams-type identity to linear codes over~$\mathbb{F}_q$.
Britz, Shiromoto and Westerb\"ack~\cite{BSW2015} generalized Greene's Theorem 
with respect to $m$-tuple weight enumerators for linear codes over finite Frobenius rings 
and, more recently, 
Chakraborty, Miezaki and Oura~\cite{CMO20xx} generalized Greene's Theorem with respect to harmonic weight enumerators.

In this paper,
we generalise all of these previous results by introducing 
harmonic $m$-tuple weight enumerators for codes over finite Frobenius rings
as well as harmonic Tutte polynomials (and, equivalently, coboundary polynomials) for demi-matroids.
We present the Greene-type identity for these enumerators and polynomials, 
and we apply this identity to obtain the MacWilliams-type identity for 
the harmonic $m$-tuple weight enumerator of a code over a finite Frobenius ring.
We also construct the relative invariant spaces that contain
the harmonic $m$-tuple weight enumerators of self-dual codes over finite fields,
in particular the four classical types of codes introduced by Gleason~\cite{Gleason}.

This paper is organized as follows.
In Section~\ref{Sec:Preli}, 
we present basic definitions and properties 
that are frequently used in this paper, 
of discrete harmonic functions, 
codes over finite Frobenius rings and 
demi-matroids; 
see in particular Theorems~\ref{thm: Bachoc iden.}, \ref{Thm:Contraction} and \ref{Thm:CodeContraction}.
In Section~\ref{Sec:MacWilliams}, 
we define the harmonic $m$-tuple weight enumerators for codes over finite Frobenius rings,
and give some formulae to compute these; 
see Theorems~\ref{Thm:NewZ} and~\ref{Thm:NewZperp}. 
Moreover, we present the MacWilliams-type identity for 
harmonic $m$-tuple weight enumerators of codes over finite Frobenius rings; see Theorem~\ref{Thm:MacWilliams}. 
In Section~\ref{Sec:Greene},
we obtain the correspondences between the Tutte polynomials (resp., coboundary polynomials) 
of demi-matroids and their duals; see Theorems~\ref{Thm:DualSupp} and~\ref{Thm:CoboundaryDualSupp}.
Moreover, we give the Greene-type identity between these polynomials 
and the harmonic $m$-tuple weight enumerators for codes over finite Frobenius rings;
see Theorem~\ref{Thm:Greene}.
As an application of this identity, 
we provide a simple proof of the MacWilliams-type identity for the
harmonic $m$-tuple weight enumerators of codes over finite Frobenius rings. 
Finally, in Section~\ref{Sec:Inv}, 
we provide the relative invariant spaces of groups corresponding
to certain type of codes; see Theorem~\ref{Thm:Invariant}. 

All computer calculations in this paper were done with the help of Magma~\cite{Magma}
and Mathematica~\cite{Mathematica}.

\section{Preliminaries}\label{Sec:Preli}

In this section, we discuss the basic definitions and notations of linear codes 
and demi-matroids that are frequently needed in this paper. 
We review~\cite{BJMS2012,BSW2015, CR1970, MS1977} for this discussion. 
Moreover, recall~\cite{Bachoc, Delsarte} for the definitions and properties of the (discrete) harmonic functions.

\subsection{Discrete harmonic functions}

Let $E := \{1,2,\ldots,n\}$ be a finite set.
We define  
$E_{d} := \{ X \subseteq E : |X| = d\}$
for $d = 0,1, \ldots, n$. 
The set of all subsets of $E$ is denoted by $2^{E}$.
We denote by 
$\R 2^{E}$ and $\R E_{d}$
the real vector spaces spanned by the elements of  
$2^{E}$ and $E_{d}$,
respectively. 
An element of 
$\R E_{d}$
is denoted by
\begin{equation}\label{Equ:FunREd}
	f :=
	\sum_{Z \in E_{d}}
	f(Z) Z
\end{equation}
and is identified with the real-valued function on 
$E_{d}$
given by 
$Z \mapsto f(Z)$. 
Such an element 
$f \in \R E_{d}$
can be extended to an element 
$\widetilde{f}\in \R 2^{E}$
by setting, for all 
$X \in 2^{E}$,
\begin{equation}\label{Equ:TildeF}
    \widetilde{f}(X)
	:=
	\sum_{Z\in E_{d}, Z\subseteq X}
	f(Z).
\end{equation}
If an element 
$g \in \R 2^{E}$
is equal to $\widetilde{f}$  
for some $f \in \R E_{d}$, 
then we say that $g$ has degree $d$. 
The differentiation $\gamma$ is the operator on $\mathbb{R}E_d$ 
defined by linearity from the identity
\begin{equation}\label{Equ:Gamma}
	\gamma(Z) := 
	\sum_{Y\in {E}_{d-1}, Y\subseteq Z} 
	Y
\end{equation}
for all 
$Z \in E_{d}$
and for all $d=0,1, \ldots n$.
Also, $\Harm_{d}(n)$ is the kernel of~$\gamma$:
\begin{equation}\label{Equ:Harm}
	\Harm_{d}(n) 
	:= 
	\ker
	\left(
	\gamma\big|_{\R E_{d}}
	\right).
\end{equation}

 \begin{rem}{\rm\cite{Bachoc,Delsarte}}\label{Rem:Gamma}
	Let $f \in \Harm_{d}(n)$ and $i \in \{0,1,\ldots,d-1\}$.
	Then 
	$\gamma^{d-i}(f) = 0$. 
    This means from definition~(\ref{Equ:Gamma})
	that
	$$\sum_{X \in E_{i}}\left(\sum_{\substack{Z \in E_{d}, X \subseteq Z}} f(Z)\right) X = 0.$$
	This implies that $\sum_{\substack{Z \in E_{d}, X \subseteq Z}} f(Z) = 0$
	for any $X \in E_{i}$.
\end{rem}

\begin{rem}\label{Rem:New}
	Let~$f \in \Harm_{d}(n)$. 
	Since $\sum_{Z \in E_{d}} f(Z) = 0$, 
	it is easy to check from~(\ref{Equ:Gamma}) that
	$\sum_{X \in E_{t}} \widetilde{f}(X) = 0$, where
	$1 \leq d \leq t \leq n$.
\end{rem}

\begin{ex}
	Let $E = \{1,2,3,4\}$ and $d =2$. 
    Let $f \in \RR E_{2}$ be the element
	\[
		f 
		= 
		a_1 \{1,2\}
		+
		a_2 \{1,3\}
		+
		a_3 \{1,4\}
		+
		a_4 \{2,3\}
		+
		a_5 \{2,4\}
		+
		a_6 \{3,4\}.
	\]
    Then 
    \[
        \gamma(f)
        =
		(a_1+a_2+a_3) \{1\}
		+
		(a_1+a_4+a_5)\{2\}
		+
		(a_2+a_4+a_6)\{3\}
		+
		(a_3+a_5+a_6)\{4\}.    
    \]
	If $X = \{1,3,4\}$,
    then $\widetilde{f}(X) = a_2+a_3+a_6$.
	Suppose that $f \in \Harm_{2}(4)$; 
    then $\gamma(f) = 0$, 
    so 
    \[
      a_1+a_2+a_3 = a_1+a_4+a_5 = a_2+a_4+a_6 = a_3+a_5+a_6 = 0.
	\]
    Solving the above equations, we can get
	\begin{multline*}
		\Harm_{2}(4)
		\ni
		f
		=
		a_1 \{1,2\}
		+
		a_2 \{1,3\}
		+
		(-a_1-a_2) \{1,4\}\\
		+
		(-a_1-a_2) \{2,3\}
		+
		a_2 \{2,4\}
		+
		a_1 \{3,4\}.		
	\end{multline*}
\end{ex}

\subsection{Linear codes}

Let $R$ be a finite ring with identity satisfying the associative property. 
Let $V:=R^{n}$ be the free $R$-module with ordinary inner product
$\bm{u}\cdot\bm{v} := u_{1}v_{1} + \cdots + u_{n}v_{n}$
for $\bm{u},\bm{v} \in V$,
where
$\bm{u} = (u_{1},\ldots,u_{n})$ and $\bm{v} = (v_{1},\ldots,v_{n})$.
For $\bm{u} \in V$, we call
$\supp(\bm{u}) := \{i \in E \mid u_{i} \neq 0\}$
the \emph{support} of $\bm{u}$,
and
$\wt(\bm{u}) := |\supp(\bm{u})|$
the \emph{weight} of $\bm{u}$. 
Similarly, the \emph{support} and \emph{weight}
of each subset $A \subseteq V$ are defined as follows:
\begin{align*}
	\Supp(A) 
	& :=
	\bigcup_{\bm{u} \in A}
	\supp(\bm{u})\\
	\wt(A) 
	& :=
	|\Supp(A)|.
\end{align*}

A \emph{left} (or \emph{right}) $R$-\emph{linear code} of length~$n$ is a left (or right) $R$-submodule of $V$.
The elements of a code are known as codewords. 
We denote the \emph{left dual code} of a right $R$-linear code~$C$ by $^{\perp}C$ 
and the \emph{right dual code} of a left $R$-linear code~$D$ by $D^{\perp}$, 
and define as follows:
\begin{align*}
	^{\perp}C
	:=
	\{
		\bm{u} \in V
		\mid
		\bm{u} \cdot \bm{v} = 0
		\mbox{ for all }
		\bm{v} \in C
	\},\\
	D^{\perp}
	:=
	\{
		\bm{u} \in V
		\mid
		\bm{v} \cdot \bm{u} = 0
		\mbox{ for all }
		\bm{v} \in D
	\}.
\end{align*}
It is immediate from the above definition that
$^\perp C$ is a left $R$-submodule of $V$,
and $D^\perp$ is a right $R$-submodule of $V$.

Let $C$ be a left (or right) $R$-linear code of length~$n$. 
Then for each subset $X \subseteq E$ with $t$ coordinates, 
the \emph{punctured code} $C\backslash X$ is the 
left (or right) $R$-submodule of $R^{n-t}$ obtained by 
deleting the coordinates $X$ from each codeword of~$C$.
Similarly, the \emph{shortened code} $C/X$ is the 
right (or left) submodule of $R^{n-t}$ obtained by 
deleting the coordinates $X$ from the codewords of $C$ which are zero on $X$. 
Define the right (or left) $R$-submodule
\[
	C(X)
	:=
	\{
	\bm{u} \in C
	\mid
	\supp(\bm{u}) \subseteq X
	\}.
\]
Note that $C(X)\backslash (E-X) = C/(E-X)$. 
Then there is an exact sequence (see~\cite{BSW2015}) of right (or left) 
$R$-modules:
\[
	0
	\longrightarrow
	C(E-X)
	\overset{\mathrm{inc}}{\longrightarrow}
	C
	\overset{\mathrm{punc}}{\longrightarrow}
	C\backslash (E-X)
	\longrightarrow
	0,
\]
where the maps $\mathrm{inc}$ and $\mathrm{punc}$ denote 
the inclusion map and puncture map, respectively. 
Therefore, we have from~\cite{BSW2015} the following lemma.

\begin{lem}\label{Lem:Lem3}
	$|C|=|C/X|\cdot |C\backslash (E-X)|$.
\end{lem}


Bachoc~\cite{Bachoc} introduced harmonic weight enumerators for binary codes 
and proved a MacWilliams-type identity for these enumerators. 
Later, this concept was extended in~\cite{BachocNonBinary,Tanabe2001} to linear codes over any finite field. 
We define the harmonic weight enumerators for linear codes over finite rings as follows.  

\begin{df}\label{DefHarmWeightBachoc}
	Let $C$ be a left (or right) $R$-linear code of length~$n$. 
    Let $f\in\Harm_{d}(n)$. 
	Then the \emph{harmonic weight enumerator} of $C$ associated to $f$ is
	defined as follows:

   \[	
	    W_{C,f}(x,y) 
	    :=
	    \sum_{X \subseteq E}
	    \widetilde{f}(X)
	    A_{C}(X)
	    x^{n-|X|}
	    y^{|X|},
    \]
    where 
    \[
    	A_{C}(X)
    	:= 
    	\#\{\bm{u} \in C \mid \supp(\bm{u}) = X\}.
    \]
\end{df}

Before presenting the definition of a finite Frobenius ring, 
we need to give some notations about modules over rings.
Let $R$ be a finite ring. 
Then we shall write $_R M$ (or $M_R$) to denote a left (or right) $R$-module $M$. 
Recall the Jacobson radical of $R$, denoted by $\textrm{Rad}(R)$,
and the socal of an $R$-module $M$, denoted by $\textrm{Soc(M)}$; 
for detailed information on $\textrm{Rad}(R)$ and $\textrm{Soc(M)}$, see~\cite{Lam1999}. 

A finite ring $R$ is called a \emph{Frobenius ring} if it satisfies the following conditions:
\begin{itemize}
	\item [(i)]
	$R/\textrm{Rad}(R) \cong \textrm{Soc}(R_R)$ as a right $R$-module;
	
	\item [(ii)]
	$R/\textrm{Rad}(R) \cong \textrm{Soc}(_R R)$ as a left $R$-module.
\end{itemize}

The proof of the following characterization of codes over finite Frobenius rings can
be found in Proposition~$1$ of~\cite{BSW2015}.

\begin{thm}
	Let $R$ be a finite ring.  
	Then the following statements are equivalent for each subset $X \subseteq E$:
	\begin{itemize}
		\item [(a)]
		$R$ is a Frobenius ring;
		
		\item [(b)]
		$|C\backslash (E-X)|\cdot |C^{\perp}/(E-X)| = |R|^{|X|}$ 
		for each left $R$-linear code~$C$;
		
		\item [(c)]
		$|C\backslash (E-X)|\cdot |^{\perp}C/(E-X)| = |R|^{|X|}$ 
		for each right $R$-linear code~$C$.
	\end{itemize}
\end{thm}

Now we present the following MacWilliams-type identity. 
We state the identity for left linear codes over finite Frobenius rings. 
Similarly, most of the results in the subsequent sections are stated for left linear codes. 
Each of these statements can be expressed equivalently for right linear codes.

\begin{thm}[MacWilliams-type identity]\label{thm: Bachoc iden.} 
	Let $R$ be a finite Frobenius ring, and $C$ be a left $R$-linear code of length~$n$.
    Let $W_{C,f}(x,y)$ be the harmonic weight enumerator of $C$ 
	associated to $f \in \Harm_{d}(n)$. 
    Then 
	\[
		W_{C,f}(x,y) 
		:= 
		(xy)^{d} Z_{C,f}(x,y),
	\]
	where $Z_{C,f}$ is a homogeneous polynomial of degree $n-2d$ 
    and satisfies
	\[
		Z_{C^{\perp},f}
		(x,y)
		= 
		(-1)^{d} 
		\frac{|R|^{d}}{|C|} 
		Z_{C,f} 
		\left( 
		x+(|R|-1)y, 
		x-y 
		\right).
	\]
\end{thm}

This is a special case of a more general MacWilliams-type identity that we will present and prove below; 
see Theorem~\ref{Thm:MacWilliams}.

\subsection{Demi-matroids}

A (real) \emph{demi-matroid} is a triple 
$D := (E, s, t)$ consisting of a set $E$ and two real-valued
functions $s, t : 2^{E} \to \RR$ satisfying the following two conditions:
\begin{itemize}
	\item [(D1)]
	if $X \subseteq Y \subseteq E$, 
	then 
	$0 \leq s(X) \leq s(Y) \leq |Y|$ 
	and $0 \leq t(X) \leq t(Y) \leq |Y|$;
	
	\item[(D2)]
	if $X \subseteq E$, 
	then 
	$|E - X| - s(E - X) = t (E) - t (X)$.
\end{itemize}
Note that 
$s(\emptyset) = t(\emptyset) = 0$ by (D1), 
so (D2) is equivalent to:

\begin{itemize}
	\item [(D3)]
	if $X \subseteq E$, then $|E - X| - t (E - X) = s(E) - s(X)$.
\end{itemize}
It is easy to see from (D3) that 
if $M = (E, \rho)$ is a matroid on $E$ with rank function~$\rho$, 
then the triple $(E, \rho, \rho^{\ast})$ is a demi-matroid, 
where $\rho^{\ast}$ is the rank function of 
the dual matroid $M^{\ast} = (E, \rho^{\ast})$.

Let $R$ be a finite ring. 
Let $C$ be a left (or right) $R$-linear code of length~$n$. 
Define two functions 
$\alpha_{C},\beta_{C}:2^{E} \to \RR$ 
as follows:

\begin{itemize}
	\item [(i)] 
	If $C$ is a left $R$-linear code, then
	\begin{align*}
		\alpha_{C}(X)
		& :=
		\log_{|R|}|C\backslash(E-X)|,\\
		\beta_{C}(X)
		& :=
		\log_{|R|}|C^{\perp}\backslash(E-X)|.
	\end{align*}

	\item [(ii)]
	If $C$ is a right $R$-linear code, then
	\begin{align*}
		\alpha_{C}(X)
		& :=
		\log_{|R|}|C\backslash(E-X)|,\\
		\beta_{C}(X)
		& :=
		\log_{|R|}|^{\perp}C\backslash(E-X)|.
	\end{align*}	
\end{itemize}

\begin{rem}\label{Rem:AlphaBeta}
	$\beta_{C}(E) = n - \alpha_{C}(E)$.
\end{rem}

Now we have the following characterization from~\cite[Theorem 1]{BSW2015} for 
codes over finite Frobenius ring.

\begin{thm}\label{Thm:CharFrobRingCodes}
	Let $R$ be a finite ring. 
    Then the following statements are equivalent:
	\begin{itemize}
		\item [(a)]
		$R$ is a Frobenius ring;
		
		\item [(b)]
		$D_{C} := (E,\alpha_{C},\beta_{C})$ is a demi-matroid for each left $R$-linear
		code~$C$;
		
		\item [(c)]
		$D_{C} := (E,\alpha_{C},\beta_{C})$ is a demi-matroid for each right $R$-linear
		code~$C$.
	\end{itemize}
\end{thm}

\begin{rem}\label{Rem:Thm2Britz}
	If $D_{C} = (E,\alpha_{C},\beta_{C})$ is demi-matroid,
	then for each $X \subseteq E$,
	\begin{itemize}
		\item [(i)]
		$|R|^{|E-X|}/|C\backslash X| = |C^{\perp}/X|$ 
		for any left $R$-linear code~$C$;
		
		\item [(ii)]
		$|R|^{|E-X|}/|C\backslash X| = |^{\perp}C/X|$ 
		for any right $R$-linear code~$C$.
	\end{itemize}
\end{rem}

Let $g : 2^{E} \to \RR$ be a real-valued function
and let $T \subseteq E$. 
Now define $\overline{g}$ and $\widehat{g}$ as follows:
\begin{align*}
	\overline{g}(X) &:= g(E) - g(E-X) \text{ for each $X \subseteq E$};\\
	\widehat{g}(X) &:= g(X\cup T) - g(T) \text{ for each $X \subseteq E-T$}.
\end{align*}
Let $D=(E,s,t)$ be a demi-matroid. 
Then the \emph{supplement} of the demi-matroid $D$ is 
$\overline{D} = (E,\overline{s},\overline{t})$, and 
the \emph{dual} 
of $D$ is $D^{\ast} = (E,t,s)$.
It is known from~\cite[Theorem 4]{BJMS2012} that $\overline{D}$
and $D^{\ast}$ are also demi-matroids and that they satisfy the properties
$(D^\ast)^\ast = D = \overline{\overline{D}}$ and $\overline{D^{\ast}} = (\overline{D})^\ast$.
Any two subsets $X$ and $Y$ such that $X \subseteq Y \subseteq E-T$ satisfy (D1). 
Also any subset $X \subseteq E-T$ satisfies (D2). 
Therefore, $(E-T,s,t)$ is a demi-matroid. 
We call this demi-matroid the \emph{deletion} of $T$ from $D$ and denoted it by $D\backslash T$.
Similarly, the \emph{contraction} of $T$ from $D$ is $(E-T,\widehat{s},\widehat{t})$ which we denote by $D/T$.

\begin{prop}
	$\overline{D^{\ast}\backslash T} = (\overline{D}\backslash T)^{\ast}$.
\end{prop}

\begin{thm}\label{Thm:Contraction}
	Let $D = (E,s,t)$ be a demi-matroid. 
    Then for each subset $T \subseteq E$, 
    $D/T = (E-T,\widehat{s},\widehat{t})$ is a demi-matroid.
\end{thm}

\begin{prop}\label{Prop:Contraction}
	$(D^{\ast}\backslash T)^{\ast} = D/T = \overline{\overline{D}\backslash T}$.
\end{prop}

Let $R$ be a finite Frobenius ring and let $C$ be a left (or right) $R$-linear code of length~$n$. 
If $D_{C} = (E,\alpha_{C},\beta_{C})$ is a demi-matroid corresponding to~$C$, 
then $\overline{D_{C}} = (E, \gamma_{C}, \delta_{C})$ (see~\cite[Theorem 2]{BSW2015}), 
where the functions $\gamma_{C},\delta_{C}:2^{E} \to \RR$ are defined as follows:
\begin{itemize}
	\item [(i)]
	If $C$ is a left $R$-linear code, then
	\begin{align*}
		\gamma_{C}(X)
		& :=
		\log_{|R|}|C/(E-X)|;\\
		\delta_{C}(X)
		& :=
		\log_{|R|}|C^{\perp}/(E-X)|.
	\end{align*}

	\item [(ii)]
	If $C$ is a right $R$-linear code, then
	\begin{align*}
		\gamma_{C}(X)
		& :=
		\log_{|R|}|C/(E-X)|;\\
		\delta_{C}(X)
		& :=
		\log_{|R|}|^{\perp}C/(E-X)|.
	\end{align*}
\end{itemize}
Moreover, $(E, \gamma_{C}, \delta_{C})$ is a demi-matroid. 
Since $\alpha_{C^{\perp}} = \beta_{C}$ and $\gamma_{C^{\perp}} = \delta_{C}$,
we have the following dual relations from~\cite{BSW2015}.

\begin{rem}
	$D_{C^{\perp}} = (D_{C})^{\ast}$ and $\overline{D}_{C^{\perp}} = (\overline{D}_{C})^{\ast}$.
\end{rem}

\begin{thm}\label{Thm:CodeContraction}
	Let $R$ be a finite Frobenius ring and
    let $C$ be a left (or right) $R$-linear code of length~$n$. 
    Then for each subset $T \subseteq E$, 
	\[
		(E-T,\widehat{\alpha}_{C},\widehat{\beta}_{C})
		=	
		D_{C/T} 
		=
		(E-T,\widehat{\gamma}_{C},\widehat{\delta}_{C})
	\]
	is a demi-matroid.
\end{thm}

\begin{proof}
	Let $X \subseteq E-T$. 
    By Lemma~\ref{Lem:Lem3},
	\begin{align*}
		\alpha_{C/T}(X)
		& =
		\log_{|R|} 
		|C/T \backslash (E-(T\cup X))|\\
		& =
		\log_{|R|} 
		\frac{|C/T|}{|C/(T\cup X)|}\\
		& =
		\log_{|R|} |C/T|
		-
		\log_{|R|} |C/(T \cup X)|\\
		& =
		\gamma_{C}(E-T)
		-
		\gamma_{C}(E-(T\cup X))\\
		& =
		\alpha_{C}(E) - \alpha_{C}(T)
		-
		\alpha_{C}(E) + \alpha_{C}(T\cup X)\\
		& =
		\alpha_{C}(T\cup X) - \alpha_{C}(T)\\
		& =
		\widehat{\alpha}_{C}(X).
	\end{align*}
	Similarly, 
	$\beta_{C/T}(X) = \widehat{\beta}_{C}(X)$.
	Therefore, 
	$D_{C/T} = (E-T,\alpha_{C/T},\beta_{C/T}) = (E-T,\widehat{\alpha}_{C},\widehat{\beta}_{C})$.
	By similar arguments, we can show that
	$D_{C/T} = (E-T,\gamma_{C/T},\delta_{C/T}) = (E-T,\widehat{\gamma}_{C},\widehat{\delta}_{C})$.
	Hence by Theorem~\ref{Thm:CharFrobRingCodes}, 
    $D_{C/T}$ is a demi-matroid.
\end{proof}

Theorem~\ref{Thm:CharFrobRingCodes} and Theorem~\ref{Thm:CodeContraction} 
together with Proposition~\ref{Prop:Contraction} imply the following corollary:

\begin{cor}
	$D_{C}/T = D_{C/T}$.
\end{cor}

\section{Harmonic generalization of MacWilliams identity}\label{Sec:MacWilliams}

In this section, 
we introduce the harmonic generalization of $m$-tuple weight enumerators for codes over finite Frobenius rings. 
First, we recall~\cite{BSW2015} for some useful notations and properties.

Let $R$ be a finite Frobenius ring and
let $C$ be a left (or right) $R$-linear code of length $n$. 
For any subset $X \subseteq E$, 
we denote 
\begin{align*}
	A_{C}^{[m]}(X)
	&:=
	\#
	\{
	(\bm{u}_{1},\ldots,\bm{u}_{m}) 
	\in 
	C^{m} 
	\mid 
	\supp(\bm{u}_{1}) \cup \cdots \cup \supp(\bm{u}_{m}) = X
	\},\\
	B_{C}^{[m]}(X)
	& :=
	\#
	\{
	(\bm{u}_{1},\ldots,\bm{u}_{m}) 
	\in 
	C^{m} 
	\mid 
	\supp(\bm{u}_{1}) \cup \cdots \cup \supp(\bm{u}_{m}) \subseteq X\},
\end{align*}
where $C^{m} := \underset{m}{\underbrace{C\times\cdots\times C}}.$ 
Then note that
$B_{C}^{[m]}(X) = |C/(E-X)|^{m}$.

\begin{rem}
	For each $X \subseteq E$, we have
	$
		B_{C}^{[m]}(X)
		=
		\sum\limits_{Y \subseteq X}	
		A_{C}^{[m]}(Y).
    $
\end{rem}

Now 
we have the following identity.

\begin{lem}[\cite{BSW2015}]\label{Lem:InEx}
	$A_{C}^{[m]}(X) = \sum\limits_{Y \subseteq X} (-1)^{|X-Y|}B_{C}^{[m]}(Y)$.
\end{lem}

\begin{df}
	Let $R$ be a finite Frobenius ring and
	let $C$ be a left (or right) $R$-linear code of length~$n$. 
	For each $f \in \Harm_{d}(n)$,
    the~$m$-\emph{tuple harmonic weight enumerator}
	of $C$ associated to $f$ is defined as follows:
	\[
		W_{C,f}^{[m]}(x,y)
		:=
		\sum_{X \subseteq E}
		\widetilde{f}(X)
		A_{C}^{[m]}(X)
		x^{|E-X|} 
		y^{|X|}.
	\]
\end{df}
The harmonic weight enumerator
$W_{C,f}(x,y)$ is obtained by 
setting $m = 1$ in the above definition. 
Now we have the following identity as a generalization of Theorem~\ref{thm: Bachoc iden.}.

\begin{thm}[MacWilliams-type identity]\label{Thm:MacWilliams}
	Let $R$ be a finite Frobenius ring. 
    Let $W_{C,f}^{[m]}(x,y)$ be the $m$-tuple harmonic weight enumerator of a left $R$-linear code $C$
	of length~$n$ associated to some $f \in \Harm_{d}(n)$.  
	Then
	\[
		W_{C,f}^{[m]}(x,y) 
		= 
		(xy)^{d} 
		Z_{C,f}^{[m]}(x,y),
	\]
	where $Z_{C,f}^{[m]}$ is a homogeneous polynomial of degree~$n-2d$
	satisfying
	\[
		Z_{C^{\perp},f}^{[m]}
		(x,y)
		=
		(-1)^{d}
		\dfrac{(|R|^m)^{d}}{|C|^{m}}
		Z_{C,f}^{[m]}
		\left(
		x+(|R|^{m}-1)y,
		x-y
		\right).
	\]
\end{thm}

We will prove this theorem at the end of this section.


\begin{lem}[\rm\cite{Bachoc}]\label{Lem:Bachoc}
	Let $f \in \Harm_{d}(n)$ and $J \subseteq E$, 
    and define
	\[
	f^{(i)}(J)
	:=
	\sum_{\substack{Z \in E_{d},\\ |J \cap Z| = i}}
	f(Z).
	\]
	Then for all 
	$0 \leq i \leq d$,
	$f^{(i)}(J) = (-1)^{d-i} \binom{d}{i} \widetilde{f}(J)$.
\end{lem}

\begin{rem}\label{Rem:BachocLem}
	From the definition of $\widetilde{f}$ 
	for $f \in \Harm_{d}(n)$,
	we have $\widetilde{f}(J) = 0$
	for each $J \in 2^{E}$ such that $|J| < d$. 
	Let $I, J \in 2^{E}$ such that $I = E - J$.
	Then
	\begin{align*}
		\widetilde{f}(J)
		=
		\sum_{\substack{Z \in E_{d},\\Z \subset J}}
		f(Z)
		=
		\sum_{\substack{Z \in E_{d},\\ |Z \cap I|=0}}
		f(Z)
		=
		f^{(0)}(I)
		=
		(-1)^{d} \widetilde{f}(E - J).		
	\end{align*}	
	We have from the above equality that if $ |J| > n-d$, 
	then $\widetilde{f}(J) = 0$. 
\end{rem}

From the above discussion and the first part of Theorem~\ref{Thm:MacWilliams}, 
we see that
\[
	Z_{C,f}^{[m]}(x,y) 
	= 
	\sum_{\substack{X \subseteq E,\\ d \le|X|\le n-d}}
	\widetilde{f}(X)
	A_{C}^{[m]}(X) 
	x^{|E-X|-d}y^{|X|-d}.
\]

\begin{thm}\label{Thm:NewZ}
	Let $R$ be a finite Frobenius ring and 
    let $C$ be a left (or right) $R$-linear code of length~$n$. 
    For each $f \in \Harm_{d}(n)$,
    \[    
		Z_{C,f}^{[m]}(x,y) 
		= 
		(-1)^{d}
		\sum\limits_{\substack{X \subseteq E,\\ d \le|X|\le n-d}}
		\widetilde{f}(X)
		(|R|^{m})^{\alpha_{C}(E)-\alpha_{C}(X)} 
		(x-y)^{|X|- d} y^{|E-X|-d}.
	\]
\end{thm}

\begin{proof}
	By Lemma~\ref{Lem:InEx} and Lemma~\ref{Lem:Bachoc},
	\begingroup
	\allowdisplaybreaks
	\begin{align*}
		&Z_{C,f}^{[m]}(x,y)\\
		& =
		\sum_{X \subseteq E}
		\widetilde{f}(X)
		A_{C}^{[m]}(X)
		x^{|E-X|-d} 
		y^{|X|-d}\\
		& =
		\sum_{X \subseteq E}
		\widetilde{f}(X)\sum_{Y \subseteq X}
		(-1)^{|X-Y|}
		B_{C}^{[m]}(Y)
		x^{|E-X|-d} 
		y^{|X|-d}\\
		& =
		\sum_{Y \subseteq E}
		B_{C}^{[m]}(Y)
		\sum_{Y \subseteq X \subseteq E}
		(-1)^{|X-Y|}\widetilde{f}(X)
		x^{|E-X|-d} 
		y^{|X|-d}\\
		& =
		\sum_{Y \subseteq E}
		B_{C}^{[m]}(Y)
		\sum_{W \subseteq E-Y}
		(-1)^{|W|}\widetilde{f}(Y\cup W)
		x^{|E-(Y\cup W)|-d} 
		y^{|Y\cup W|-d}\\
		& =
		\sum_{Y \subseteq E}
		B_{C}^{[m]}(Y)
		\sum_{i=0}^{|E-Y|-d}\sum_{\substack{W \subseteq E-Y,\\|W|=i}}
		(-1)^{i}\widetilde{f}(Y\cup W)
		x^{|E-Y|-d-i} 
		y^{|Y|-d+i}\\
		& =
		\sum_{Y \subseteq E}
		B_{C}^{[m]}(Y) y^{|Y|-d}
		\sum_{i=0}^{|E-Y|-d}\sum_{\substack{W \subseteq E-Y,\\|W|=i}}
		\widetilde{f}(Y\cup W)
		x^{|E-Y|-d-i} 
		(-y)^{i}\\
		& =
		\sum_{Y \subseteq E}
		B_{C}^{[m]}(Y) y^{|Y|-d}
		\sum_{i=0}^{|E-Y|-d}\sum_{\substack{W \subseteq E-Y,\\|W|=i}}
		\sum_{\substack{Z\in E_d,\\Z\subseteq Y\cup W}} f(Z)
		x^{|E-Y|-d-i} 
		(-y)^{i}\\
		& =
		\sum_{Y \subseteq E}
		B_{C}^{[m]}(Y) y^{|Y|-d}
		\sum_{i=0}^{|E-Y|-d}\sum_{Z\in E_d}
		\sum_{\substack{W\subseteq E-Y,\\Z\cap(E-Y)\subseteq W,\\|W|=i}} f(Z)
		x^{|E-Y|-d-i} 
		(-y)^{i}\\
		& =
		\sum_{Y \subseteq E}
		B_{C}^{[m]}(Y) y^{|Y|-d}
		\sum_{i=0}^{|E-Y|-d}\sum_{j=0}^d\sum_{\substack{Z\in E_d,\\|Y\cap Z|=j}}
		\sum_{\substack{W\subseteq E-Y,\\Z\cap(E-Y)\subseteq W,\\|W|=i}} f(Z)
		x^{|E-Y|-d-i} 
		(-y)^{i}\\
		& =
		\sum_{Y \subseteq E}
		B_{C}^{[m]}(Y) y^{|Y|-d}
		\sum_{i=0}^{|E-Y|-d}\sum_{j=0}^d
		\binom{|E-Y|-(d-j)}{i-(d-j)} \sum_{\substack{Z\in E_d,\\|Y\cap Z|=j}}f(Z)
		x^{|E-Y|-d-i} 
		(-y)^{i}\\
		& =
		\sum_{Y \subseteq E}
		B_{C}^{[m]}(Y) y^{|Y|-d}
		\sum_{i=0}^{|E-Y|-d}\sum_{j=0}^d
		\binom{|E-Y|-d+j}{|E-Y|-i} f^{(j)}(Y)
		x^{|E-Y|-d-i} 
		(-y)^{i}\\
		& =
		\sum_{Y \subseteq E}
		B_{C}^{[m]}(Y) y^{|Y|-d}
		\sum_{i=0}^{|E-Y|-d}\sum_{j=0}^d
		\binom{|E-Y|-d+j}{|E-Y|-i}(-1)^{d-j}\binom{d}{j}\widetilde{f}(Y)
		x^{|E-Y|-d-i} 
		(-y)^{i}\\
		& =
		\sum_{Y \subseteq E}
		\widetilde{f}(Y)
		B_{C}^{[m]}(Y)
		y^{|Y|-d}
		\sum_{i = 0}^{|E-Y|-d}
		\sum_{j=0}^d
		(-1)^{d-j}
		\binom{|E-Y|-d+j}{|E-Y|-i}\binom{d}{j}
		x^{(|E-Y|-d)-i}
		(-y)^{i}\\
		& =
		\sum_{\substack{Y \subseteq E,\\ d \le|Y|\le n-d}}
		\widetilde{f}(Y)
		B_{C}^{[m]}(Y)
		y^{|Y|-d}
		\sum_{i = 0}^{|E-Y|-d}
		\binom{|E-Y|-d}{i}
		x^{(|E-Y|-d)-i}
		(-y)^{i}\\
		& =
		\sum_{\substack{Y \subseteq E,\\ d \le|Y|\le n-d}}
		\widetilde{f}(Y)
		B_{C}^{[m]}(Y)
		(x-y)^{|E-Y|-d}
		y^{|Y|-d}\\
		& =
		\sum_{\substack{Y \subseteq E,\\ d \le|Y|\le n-d}}
		\widetilde{f}(Y)
		|C/(E-Y)|^{m}
		(x-y)^{|E-Y|-d}
		y^{|Y|-d}\\
		& = 
		(-1)^{d}
		\sum_{\substack{Y \subseteq E,\\ d \le|Y|\le n-d}}
		\widetilde{f}(Y)
		|C/Y|^{m}
		(x-y)^{|Y|-d}
		y^{|E-Y|-d}\\
		& =
		(-1)^{d}
		\sum_{\substack{Y \subseteq E,\\ d \le|Y|\le n-d}}
		\widetilde{f}(Y)
		(|R|^{m})^{\gamma_{C}(E-Y)}
		(x-y)^{|Y|-d}
		y^{|E-Y|-d}\\
		& = 
		(-1)^{d}
		\sum_{\substack{Y \subseteq E,\\ d \le|Y|\le n-d}}
		\widetilde{f}(Y)
		(|R|^{m})^{\alpha_{C}(E)-\alpha_{C}(Y)}
		(x-y)^{|Y|-d}
		y^{|E-Y|-d}.
	\end{align*}
	\endgroup
	This completes the proof.
\end{proof}

In the proof above, we use a binomial identity as follows:
\[\sum_{j=0}^d
(-1)^{d-j}\binom{|E-Y|-d+j}{|E-Y|-i}\binom{d}{j}=\binom{|E-Y|-d}{i}.\]    
\begin{proof}
	For any polynomial $p(x,y)$, let $\left[x^m y^n\right]p(x,y)$ denote the coefficient of $x^m y^n$ in the $p(x,y)$. Notice that
	\begin{align*}
		\binom{|E-Y|-d}{i}&=\left[x^{|E-Y|-i}y^i\right](x+y)^{|E-Y|-d}x^d\\
		&=\left[x^{|E-Y|-i}y^i\right](x+y)^{|E-Y|-d}(x+y-y)^d\\
		&=\left[x^{|E-Y|-i}y^i\right]\sum_{j=0}^d\binom{d}{j}(x+y)^{|E-Y|-d+j}(-y)^{d-j}\\
		&=\sum_{j=0}^d\binom{d}{j}\left[x^{|E-Y|-i}y^i\right](x+y)^{|E-Y|-d+j}(-y)^{d-j}\\
		&=\sum_{j=0}^d
		(-1)^{d-j}\binom{|E-Y|-d+j}{|E-Y|-i}\binom{d}{j}.
	\end{align*}
\end{proof}

\begin{thm}\label{Thm:NewZperp}
	Let $R$ be a finite Frobenius ring and let $C$ be a left $R$-linear code of length~$n$. 
	Let $f \in \Harm_{d}(n)$.
	Then
	\[
		Z_{C^{\perp},f}^{[m]}(x,y) 
		= 
		\sum\limits_{\substack{X \subseteq E,\\ d \le|X|\le n-d}}
		\widetilde{f}(X)
		(|R|^{m})^{d-\alpha_{C}(X)} 
		(x-y)^{|E-X|- d} 
		(|R|^m y)^{|X|-d}.
	\]
\end{thm}

\begin{proof}
	By Theorem~\ref{Thm:NewZ} and Remark~\ref{Rem:Thm2Britz},
    we see that
	\begin{align*}
		Z_{C^{\perp},f}^{[m]}(x,y) 
		& =
		(-1)^{d}
		\sum_{\substack{X \subseteq E,\\ d \le|X|\le n-d}}
		\widetilde{f}(X)
		|C^{\perp}/X|^{m}
		(x-y)^{|X|-d}
		y^{|E-X|-d}\\
		& =
		(-1)^{d}
		\sum_{\substack{X \subseteq E,\\ d \le|X|\le n-d}}
		\widetilde{f}(X)
		\left(
		\frac{|R|^{|E-X|}}{|C\backslash X|}
		\right)^m
		(x-y)^{|X|-d}
		y^{|E-X|-d}\\
		& =
		(-1)^{d}
		\sum_{\substack{X \subseteq E,\\ d \le|X|\le n-d}}
		\widetilde{f}(X)
		\bigg(
		\frac{|R|^{|E-X|-d+d}}{|R|^{\alpha_{C}(E-X)}}
		\bigg)^m
		(x-y)^{|X|-d}
		y^{|E-X|-d}\\
		& =
		(-1)^{d}
		\sum_{\substack{X \subseteq E,\\ d \le|X|\le n-d}}
		\widetilde{f}(X)
		(|R|^m)^{d-\alpha_{C}(E-X)}
		(x-y)^{|X|-d}
		(|R|^{m}y)^{|E-X|-d}\\
		& =
		\sum_{\substack{X \subseteq E,\\ d \le|X|\le n-d}}
		\widetilde{f}(X)
		(|R|^m)^{d-\alpha_{C}(X)}
		(x-y)^{|E-X|-d}
		(|R|^{m}y)^{|X|-d}.
	\end{align*}
This completes the proof.
\end{proof}


\begin{proof}[Proof of Theorem~\ref{Thm:MacWilliams}]
	The technical Lemma~\ref{Lem:Bachoc} and Remark~\ref{Rem:BachocLem} show that $Z_{C,f}^{[m]}(x,y)$ is a polynomial. 
	By Theorems~\ref{Thm:NewZperp} {and~\ref{Thm:NewZ}}, 
	\begin{align*}
		Z_{C^{\perp},f}^{[m]}(x,y) 
		& = 
		\sum\limits_{\substack{X \subseteq E,\\ d \le|X|\le n-d}}
		\widetilde{f}(X)
		(|R|^{m})^{d-\alpha_{C}(X)} 
		(x-y)^{|E-X|- d} 
		(|R|^m y)^{|X|-d}\\
		& =
		(-1)^{d}
		\frac{(|R|^m)^{d}}{(|R|^m)^{\alpha_{C}(E)}}\\
		& \times
			(-1)^{d}
			\sum\limits_{\substack{X \subseteq E,\\ d \le|X|\le n-d}}
			\widetilde{f}(X)
			(|R|^{m})^{\alpha_{C}(E)-\alpha_{C}(X)}  
			(|R|^m y)^{|X|-d}
			(x-y)^{|E-X|- d}
            \\
		& =
		(-1)^{d}
		\frac{(|R|^{m})^{d}}{|C|^{m}}
		Z_{C,f}^{[m]}
		\left(
			{x+(|R|^m-1)y},
			{x-y}
		\right).
	\end{align*}
Hence the proof is completed.
\end{proof}

\section{Harmonic generalization of Greene's theorem}\label{Sec:Greene}

Crapo~\cite{Crapo} presented two matroid polynomials,
namely the Tutte polynomial and the coboundary polynomial.
The Tutte polynomial was originally introduced for graphs by Tutte~\cite{Tutte1954,Tutte1967} 
who called it the dichromatic polynomial. 
Coboundary polynomials for matroids are equivalent to Tutte polynomials.
For a detailed discussions of these polynomials, we refer the reader to~\cite{BrCa22}.
In this section, 
we define the Tutte polynomial and the coboundary polynomial of 
a demi-matroid associated to a discrete harmonic function, and discuss some of their properties. 
Finally, we give the harmonic analogue of Greene's Theorem. 

\begin{df}
	Let $D = (E,s,t)$ be a demi-matroid and 
	$f$ be a harmonic function with degree~$d$.
	Then the \emph{harmonic Tutte polynomial} of $D$
	associated with $f$ is defined as:
	\[
	T(D,f;x,y)
	:=
	\sum_{X \subseteq E}
	\widetilde{f}(X)
	(x-1)^{s(E)-s(X)}
	(y-1)^{|X|-s(X)}.
	\]
\end{df}

\begin{thm}\label{Thm:DualSupp}
	Let $D = (E,s,t)$ be a demi-matroid. 
    Let $f \in \Harm_{d}(n)$. 
	Then we have
	\begin{itemize}
		\item [(i)] 
		${T(D^\ast,f;x,y)} = (-1)^{d}T(D,f;y,x)$, and
		\item [(ii)]
		$T(\overline{D},f;x,y) = (-1)^{d}(x-1)^{s(E)} (y-1)^{t(E)} T(D,f;\dfrac{x}{x-1},\dfrac{y}{y-1})$.
	\end{itemize} 
\end{thm}

\begin{proof}
	(i) Since $D = (E,s,t)$ is a demi-matroid, 
	the dual of $D$ is ${D^{\ast} = (E,t,s)}$. 
	Now using Condition~(D2) and Remark~\ref{Rem:BachocLem}, 
	we can write
	\begin{align*}
		{T(D^\ast,f;x,y)}
		& =
		\sum_{X \subseteq E}
		\widetilde{f}(X)
		(x-1)^{t(E)-t(X)}
		(y-1)^{|X|-t(X)}\\
		& =
		\sum_{X \subseteq E}
		\widetilde{f}(X)
		(x-1)^{|E-X|-s(E-X)}
		(y-1)^{s(E)-s(E-X)}\\
		& =
		(-1)^d
		\sum_{X \subseteq E}
		\widetilde{f}(X)
		(y-1)^{s(E)-s(X)}
		(x-1)^{|X|-s(X)}\\
		& =
		(-1)^{d}
		T(D,f;y,x).
	\end{align*}
	(ii) Since $D = (E,s,t)$ is a demi-matroid, 
	it follows that $\overline{D} = (E,\overline{s}, \overline{t})$. 
	Therefore using Condition~(D2), (D3) and Remark~\ref{Rem:BachocLem}, 
	we can write
        \begingroup
        \allowdisplaybreaks
	\begin{align*}
		T(\overline{D},f;x,y)
		& =
		\sum_{X \subseteq E}
		\widetilde{f}(X)
		(x-1)^{\overline{s}(E)-\overline{s}(X)}
		(y-1)^{|X|-\overline{s}(X)}\\
		& =
		\sum_{X \subseteq E}
		\widetilde{f}(X)
		(x-1)^{{s}(E-X)}
		(y-1)^{|X|-s(E)+{s}(E-X)}\\
		& =
		(-1)^{d}
		\sum_{X \subseteq E}
		\widetilde{f}(X)
		(x-1)^{{s}(X)}
		(y-1)^{|E-X|-s(E)+{s}(X)}\\
		& =
		(-1)^{d}
		\sum_{X \subseteq E}
		\widetilde{f}(X)
		(x-1)^{{s}(X)}
		(y-1)^{t(E-X)}\\
		& =
		(-1)^{d}
		(x-1)^{s(E)}
		(y-1)^{t(E)}
		\sum_{X \subseteq E}
		\widetilde{f}(X)
		(x-1)^{-s(E)+{s}(X)}
		(y-1)^{-t(E)+t(E-X)}\\
		& =
		(-1)^{d}
		(x-1)^{s(E)}
		(y-1)^{t(E)}
		\sum_{X \subseteq E}
		\widetilde{f}(X)
		(x-1)^{-s(E)+{s}(X)}
		(y-1)^{-|X|+s(X)}\\
		& =
		(-1)^{d}
		(x-1)^{s(E)}
		(y-1)^{t(E)}
		T(D,f;\frac{x}{x-1},\frac{y}{y-1}).
	\end{align*}
        \endgroup
	This completes the proof.
\end{proof}

\begin{df}
	Let $D = (E,s,t)$ be a demi-matroid and $f$ be a harmonic function with degree~$d$.
	Then the \emph{harmonic coboundary polynomial} of $D$ associated to $f$ is
	defined as follows:
	\[
		W_{D,f}(\lambda,x,y)
		:=
		\sum_{T\subseteq E}
		\widetilde{f}(T)
		\chi(D.T ;\lambda)
		x^{|E-T|} y^{|T|},
	\]
	where 
	$\chi(D;\lambda) := \sum_{X \subseteq E} (-1)^{|X|} \lambda^{s(E)-s(X)}$
	is the \emph{characteristic polynomial} of $D$, and 
	$D.T := D/(E-T)$.
\end{df}

\begin{rem}
	$\chi(D.\emptyset; \lambda) = 1$.
\end{rem}

\begin{rem}\label{Rem:ZDf}
	Let $W_{D,f}(\lambda,x,y)$ be the harmonic coboundary polynomial of a 
	demi-matroid $D$ on $E$ associated to $f \in \Harm_{d}(n)$. 
	Then Lemma~\ref{Lem:Bachoc} and Remark~\ref{Rem:BachocLem} imply that
	\[
		W_{D,f}(\lambda,x,y) = (xy)^{d} Z_{D,f}(\lambda,x,y),
	\]
	where $Z_{D,f}(\lambda,x,y)$ is a  homogeneous polynomial of degree $n-2d$.
\end{rem}

\begin{thm}\label{Thm:AlterCoboundary}
	$Z_{D,f}(\lambda,x,y) = (-1)^{d}\sum\limits_{\substack{T\subseteq E,\\d \leq |T|\leq n-d}} \widetilde{f}(T)\lambda^{s(E)-s(T)} (x-y)^{|T|-d} y^{|E-T|-d}$.
\end{thm}

\begin{proof}
	From Remark~\ref{Rem:ZDf} we have
	\begin{align*}
		Z_{D,f}(\lambda,x,y)
		& =
		\sum_{\substack{T\subseteq E,\\d \leq |T|\leq n-d}}
		\widetilde{f}(T)
		\chi(D.T ;\lambda)
		x^{|E-T|-d} y^{|T|-d}\\
		& =
		\sum_{\substack{T\subseteq E,\\d \leq |T|\leq n-d}}
		\widetilde{f}(T)
		\left(
		\sum_{X \subseteq T} 
		(-1)^{|X|} 
		\lambda^{\widehat{s}(T)-\widehat{s}(X)}
		\right)
		x^{|E-T|-d} y^{|T|-d}\\
		& =
		\sum_{\substack{T\subseteq E,\\d \leq |T|\leq n-d}}
		\widetilde{f}(T)
		\left(
		\sum_{X \subseteq T} 
		(-1)^{|X|} 
		\lambda^{s(E)-s(X\cup(E-T))}
		\right)
		x^{|E-T|-d} y^{|T|-d}.		
	\end{align*}
	Substituting $X \cup (E-T)$ by $T$ in the first sum, we obtain
	\begin{align*}
		Z_{D,f}(\lambda,x,y)
		& =
		(-1)^{d}
		\sum_{\substack{T\subseteq E,\\d \leq |T|\leq n-d}}
		\widetilde{f}(T)
		\lambda^{s(E)-s(T)}
		y^{|E-T|-d}
		\left(
		\sum_{\substack{X \subseteq T,\\ 0 \leq |X| \leq |T|-d}} 
		x^{(|T|-d)-|X|}
		(-y)^{|X|}
		\right)\\
		& =
		(-1)^{d}
		\sum_{\substack{T\subseteq E,\\d \leq |T|\leq n-d}}
		\widetilde{f}(T)
		\lambda^{s(E)-s(T)}
		y^{|E-T|-d}
		\left(
		\sum_{i = 0}^{|T|-d}
		\binom{|T|-d}{i}
		x^{(|T|-d)-i}
		(-y)^{i}
		\right)\\
		& =
		(-1)^{d}
		\sum_{\substack{T\subseteq E,\\d \leq |T|\leq n-d}}
		\widetilde{f}(T)
		\lambda^{s(E)-s(T)}
		y^{|E-T|-d}
		(x-y)^{|T|-d}.	
	\end{align*}
	This completes the proof.
\end{proof}

The following proposition via Theorem~\ref{Thm:AlterCoboundary} shows that 
the harmonic coboundary polynomial is equivalent to the harmonic Tutte polynomial.

\begin{prop}
	For a demi-matroid $D = (E,s,t)$,
	\[
		(-1)^{d}
		(y-1)^{s(E)-d}
		T(D,f;x,y)
		=
		Z_{D,f}((x-1)(y-1),y,1).
	\]
\end{prop}

The harmonic coboundary polynomial translation of Theorem~\ref{Thm:DualSupp} is 
as follows. Since the proof of the theorem is straightforward, we omit it.

\begin{thm}\label{Thm:CoboundaryDualSupp}
	Let $D = (E,s,t)$ be a demi-matroid. Let $f \in \Harm_{d}(n)$. 
	Then we have
	\begin{itemize}
		\item [(i)] 
		$\lambda^{s(E)-d} Z_{D^\ast,f}(\lambda,x,y) = (-1)^{d} Z_{D,f}(\lambda, x+(\lambda-1)y, x-y)$,
		\item [(ii)]
		$\lambda^{-s(E)} Z_{\overline{D},f}(\lambda,x,y) = (-1)^{d} Z_{D,f}(1/\lambda,x,x-y)$.
	\end{itemize} 
\end{thm}

The harmonic generalization of Greene's Theorem was given in~\cite[Theorem 4.1]{CMO20xx}. 
In the following theorem, we give the demi-matroid analogue of~\cite[Theorem 4.1]{CMO20xx} 
that gives the representation of 
the $m$-tuple harmonic weight enumerators of linear codes over finite Frobenius ring 
in terms of harmonic coboundary polynomials of demi-matroids.

\begin{thm}[Greene-type identity]\label{Thm:Greene}
	Let $R$ be a finite Frobenius ring. 
	Let $D_{C} =(E,\alpha_{C},\beta_{C})$ be the demi-matroid corresponding
	to left (or right) $R$-linear code~$C$ of length~$n$. Let $f \in \Harm_{d}(n)$. 
	Then we have the following relation:
	\[
		Z_{C,f}^{[m]}(x,y) 
		= 
		Z_{D_{C},f}(|R|^{m},x,y).
	\]
\end{thm}

\begin{proof}
	Theorem~\ref{Thm:NewZ} and Theorem~\ref{Thm:AlterCoboundary} completes the proof.
\end{proof}

Greene's Theorem is more familiar in terms of Tutte polynomials. So, it is natural to restate the above Greene-type identity in terms of the harmonic Tutte polynomials of demi-matroids.

\begin{cor}\label{Cor:Greene}
	\[
		Z_{C,f}^{[m]}(x,y)
		=
		(-1)^{d}
		(x-y)^{\alpha_{C}(E)-d}
		y^{n-\alpha_{C}(E)-d}
		T
		\left(
		D_{C},f;
		\dfrac{x+(|R|^{m}-1)y}{x-y},
		\dfrac{x}{y}
		\right).
	\]
\end{cor}

\begin{proof}
	By Theorem~\ref{Thm:NewZ} 
	and	Remark~\ref{Rem:BachocLem}, we can write
	\begin{align*}
		Z_{C,f}^{[m]}(x,y)
		& = 
		(-1)^{d}
		\sum\limits_{\substack{X \subseteq E,\\ d \le|X|\le n-d}}
		\widetilde{f}(X)
		(|R|^{m})^{\alpha_{C}(E)-\alpha_{C}(X)} 
		(x-y)^{|X|- d} y^{|E-X|-d}\\
		& = 
		(-1)^{d}
		(x-y)^{\alpha_{C}(E)-d}
		y^{|E|-\alpha_{C}(E)-d}\\
		& \quad
		\sum\limits_{\substack{X \subseteq E,\\ d \le|X|\le n-d}}
		\widetilde{f}(X)
		\left(
			\frac{|R|^m y}{x-y}
		\right)^{\alpha_{C}(E)-\alpha_{C}(X)} 
		\left(
			\frac{x-y}{y}
		\right)^{|X|-\alpha_{C}(X)}\\
		& =
		(-1)^{d}
		(x-y)^{\alpha_{C}(E)-d}
		y^{n-\alpha_{C}(E)-d}
		T
		\left(
		D_{C},f;
		\dfrac{x+(|R|^{m}-1)y}{x-y},
		\dfrac{x}{y}
		\right).
	\end{align*}
Hence we have the identity.
\end{proof}



\begin{ex}
	Let 
	$E = \{1,2,3\}$ and
	$f = a\{1\} + b\{2\} - (a+b)\{3\}$
	be a harmonic function of degree $d = 1$.
	Let $C$ be a linear code over $\ZZ_{4}$ with generator matrix
	\[
		\begin{pmatrix}
			1 & 1 & 0\\
			0 & 0 & 3
		\end{pmatrix}.
	\]
	Then by direct calculation, we get the harmonic weight enumerator of~$C$ to be
	\begin{align*}
		W_{C,f}(x,y)
		& =
		-3(a+b)x^2 y 
		+3(a+b)xy^2\\
		& =
		\underset{Z_{C,f}(x,y)}{\underbrace{3(a+b)(y-x)}} xy
	\end{align*}

The harmonic Tutte polynomial and harmonic coboundary polynomial of the demi-matroid $D_{C} = (E,\alpha_{C},\beta_{C})$ are as follows:

\begin{align*}
	T(D_{C},f;x,y)
	& =
	\sum_{X \subseteq E}
	\widetilde{f}(X)
	(x-1)^{\alpha_{C}(E)-\alpha_{C}(X)}
	(y-1)^{|X|-\alpha_{C}(X)}\\
	& =
	(a+b)
	\{
	(x-1)(y-1) -1
	\}.\\
	W_{D_{C},f}(\lambda,x,y)
	& =
	\sum_{X \subseteq E}
	\widetilde{f}(X)
	\chi(D_{C}.X;\lambda)
	x^{|E-X|}
	y^{|X|}\\
	& =
	(a+b) (\lambda-1) xy^2
	-
	(a+b) (\lambda-1) x^2y\\
	& =
	\underset{Z_{C,f}(x,y)}{\underbrace{(a+b)(\lambda-1)(y-x)}} xy .
\end{align*}
It follows that
\[
	(-1)^{1}
	(x-y)^{1}
	y^{0}
	T\left(D_{C},f;\frac{x+3y}{x-y},\frac{x}{y}\right)
	=
	Z_{C,f}(x,y)
	=
	Z_{D_{C},f}(4,x,y),
\]
as in Theorem~\ref{Thm:Greene} and in Corollary~\ref{Cor:Greene}
\end{ex}

As an application of Theorem~\ref{Thm:Greene}, 
we now give {a very} simple alternative proof of the 
MacWilliams-type identity stated in Theorem~\ref{Thm:MacWilliams}.

\begin{proof}[Alternative proof of Theorem~\ref{Thm:MacWilliams}]
	Let $C$ be a left $R$-linear code of length~$n$,
	and $D_{C} = (E,\alpha_{C},\beta_{C})$ be its demi-matroid.
	{Then by Remark 2.5, we have
	$D_{C^{\perp}} = (D_{C})^{\ast} = (E,\beta_{C},\alpha_{C})$.
	Theorems~\ref{Thm:CoboundaryDualSupp} and~\ref{Thm:Greene} imply that}
    \begin{align*}
		Z_{C^{\perp},f}^{[m]}(x,y)
		& =
		Z_{D_{C^{\perp}},f}(|R|^{m},x,y)\\
		& =
		Z_{(D_{C})^{\ast},f}(|R|^{m},x,y)\\
		& =
		(-1)^{d}
		(|R|^{m})^{-\alpha_{C}(E) + d}
		Z_{D_{C},f}(|R|^m, x+(|R|^{m}-1)y,x-y)\\
		& =
		(-1)^{d}
		\frac{(|R|^{m})^d}{|C|^{m}}
		Z_{C,f}^{[m]}(x+(|R|^{m}-1)y,x-y).
	\end{align*}
	Hence, the theorem is proved. 
\end{proof}

\section{Invariant theory}\label{Sec:Inv}

In this section, we consider linear codes over finite fields. Let $\FF_{q}$ be a 
finite field of order~$q$, where $q$ is a prime power.
A linear code $C$ of length $n$ is a linear subspace of $\FF_{q}^{n}$ with ordinary
inner product: 
$\bm{u} \cdot \bm{v}
:=
u_{1}v_{1} + \cdots + u_{n}v_{n}$. 
If $q$ is an even power of an arbitrary prime~$p$,
then it is convenient to consider another inner product
given by
$
\bm{u} \cdot \bm{v}
:=
u_{1}\overline{v_{1}} + \cdots + u_{n}\overline{v_{n}},
$
where $\overline{v_{i}}:= {v_{i}}^{\sqrt{q}}$.
Then the dual $C^{\perp}$ of a linear code $C$ is defined as follows: 
$C^{\perp}=\{ \bm{v}\in \FF_{q}^{n}\ \mid \ \bm{u} \cdot \bm{v} = 0 \mbox{ for all }\bm{x}\in C\}$. Now we have the following $\FF_{q}$-analogue of the MacWilliams identity stated in Theorem~\ref{Thm:MacWilliams}.

\begin{thm}[MacWilliams-type identity]
	Let $W_{C,f}^{[m]}(x,y)$ be the $m$-tuple harmonic weight enumerator of an $\FF_{q}$-linear code $C$ of length~$n$ associated to $f \in \Harm_{d}(n)$.  
	Then
	\[
	W_{C,f}^{[m]}(x,y) 
	= 
	(xy)^{d} 
	Z_{C,f}^{[m]}(x,y),
	\]
	where $Z_{C,f}^{[m]}$ is a homogeneous polynomial of degree~$n-2d$,
	satisfying
	\[
	Z_{C^{\perp},f}^{[m]}
	(x,y)
	=
	(-1)^{d}
	\dfrac{(q^{n/2})^{m}}{|C|^{m}}
	Z_{C,f}^{[m]}
	\left(
	\dfrac{x+(q^{m}-1)y}{q^{m/2}},
	\dfrac{x-y}{q^{m/2}}
	\right).
	\]
\end{thm}

A linear code $C$ is called self-dual if $C=C^{\perp}$. 
In this section, we consider the following self-dual codes (see~\cite{CS1999,HP2003}): 
\begin{tabbing}
	Type I: A code is defined over $\FF_{2}^{n}$ with all weights divisible by $2$,\\
	Type II: A code is defined over $\FF_{2}^{n}$ with all weights divisible by $4$,\\
	Type III: A code is defined over $\FF_{3}^{n}$ with all weights divisible by $3$,\\
	Type IV: A code is defined over $\FF_{4}^{n}$ with all weights divisible by $2$. \
\end{tabbing}
For detailed expressions of these codes, see~\cite{BMS1972, Gleason, MMS1972, NRS}.


Let $X\in \{\textrm{I},\textrm{II},\textrm{III},\textrm{IV}\}$ and 
$C$ be a Type~X code of length~$n$. 
In this section, 
we show that for Type X codes, 
$Z_{C,f}^{[m]}(x,y)$ is a relative invariant for a group and 
its character. Moreover, 
we give explicit generators of such relative invariant spaces. 

Let
\[
S_{m,q}
= 
\dfrac{1}{\sqrt{q^{m}}}
\begin{pmatrix}
	1 & (q^{m}-1)\\
	1 & -1
\end{pmatrix}.
\]
We consider the groups 
\begin{enumerate}
	\item 
	$G_{m}^{\textrm{I}} = \langle S_{m,2},\omega_2 I \rangle$,
	
	\item 
	
	$G_{m}^{\textrm{II}} = \langle S_{m,2}, \omega_8 I\rangle$, 
	\item 
	
	$G_{m}^{\textrm{III}} = \langle S_{m,3}, \omega_4 I\rangle$, 
	\item 
	
	$G_{m}^{\textrm{IV}} = \langle S_{m,4}, \omega_2 I\rangle$, 
\end{enumerate}
together with the character
\begin{enumerate}
	\item 
	$\chi^{\textrm{I}}_{d}(S_{m,2}) := (-1)^{-d},\ \chi^{\rm{I}}_{d}(\omega_2 I) := (\omega_2)^{-d}$,
	\item 
	$\chi^{\textrm{II}}_{d}(S_{m,2}) := (-1)^{-d},\ \chi^{\rm{II}}_{d}(\omega_8 I) := (\omega_8)^{-d}$,
	\item 
	$\chi^{\textrm{III}}_{d}(S_{m,3}) := (-1)^{-d},\ \chi^{\rm{III}}_{d}(\omega_4 I) := (\omega_4)^{-d}$,
	\item 
	$\chi^{\textrm{IV}}_{d}(S_{m,4}) := (-1)^{-d},\ 
	\chi^{\rm{IV}}_{d}(\omega_2 I) := (\omega_2)^{-d}$, 
\end{enumerate}
where $I$ is the identity matrix, and 
$\omega_k=\exp(2\pi i/k)$. 

Let $X\in \{\textrm{I},\textrm{II},\textrm{III},\textrm{IV}\}$ and 
$C$ be a Type~X code of length~$n$. 
Then for $f \in \Harm_{d}(n)$, 
its $m$-tuple Hamming weight enumerator associated to $f$ 
is written by 
\[
W_{C,f}^{[m]}(x,y) = (xy)^{d} Z_{C,f}^{[m]}(x,y). 
\]
Then by Theorem \ref{Thm:MacWilliams}, 
$Z_{C,f}^{[m]}(x,y)$ is a relative invariant of $G^\textrm{X}$ with respect to $\chi^\textrm{X}_d$: 
\[
Z_{C,f}^{[m]}(g(x,y)^{T}) 
= 
\chi^\textrm{X}_{d}(g) Z_{C,f}^{[m]}(x,y).
\]
Hence, we have 
\[
Z_{C,f}^{[m]}(x,y)\in I_{G^{\rm{X}},\chi^{\rm{X}}_d}, 
\]
where $I_{G^{\rm{X}},\chi^{\rm{X}}_d}$ is the space of relative invariants 
of $G^\textrm{X}$ with respect to $\chi^\textrm{X}_d$: 
\[
I_{G^{\rm{X}},\chi^{\rm{X}}_d}=\{P(x,y)\in \CC[x,y]\mid g.P=\chi^{\rm{X}}_d(g)P,\ 
\forall g\in G^{\rm{X}}\}, 
\]
where $(g.P)(x)=P(gx^T)$. 
In the following theorem, we give explicit generators of 
$I_{G^{\rm{X}},\chi^{\rm{X}}_d}$. 

\begin{thm}\label{Thm:Invariant}
	\begin{enumerate}
		\item
		Let $\mathfrak{R}^{\rm{I}}=\CC[P^{\rm{I}}_{8,1},P^{\rm{I}}_{8,2}]$. Then we have 
		\[
		I_{G_{m}^{\rm{I}},\chi^{\rm{I}}_d}=
		\begin{cases}
			&\mathfrak{R}^{\rm{I}}\ (k\equiv 0\pmod{2}),\\
			&Q^{\rm{I}}_{1,1,1}\mathfrak{R}^{\rm{I}}\ (k\equiv 1\pmod{2}),
		\end{cases}
		\]
		where the $P^{\rm{I}}_{\ast,\ast}$ and $Q^{\rm{I}}_{\ast,\ast,\ast}$ are 
		listed in~\cite{miezakiweb}.
		
		\item 
		Let $\mathfrak{R}^{\rm{II}}=\CC[P^{\rm{II}}_{8,1},P^{\rm{II}}_{8,2}]$. Then we have 
		\[
		I_{G_{m}^{\rm{II}},\chi^{\rm{II}}_d}=
		\begin{cases}
			&\mathfrak{R}^{\rm{II}}\oplus
			Q^{\rm{II}}_{8,1,0}\mathfrak{R}^{\rm{II}}\oplus
			Q^{\rm{II}}_{8,2,0}\mathfrak{R}^{\rm{II}}\oplus
			Q^{\rm{II}}_{8,3,0}\mathfrak{R}^{\rm{II}}\ (k\equiv 0\pmod{8}),\\
			&Q^{\rm{II}}_{7,1,1}\mathfrak{R}^{\rm{II}}\oplus
			Q^{\rm{II}}_{7,2,1}\mathfrak{R}^{\rm{II}}\oplus
			Q^{\rm{II}}_{7,3,1}\mathfrak{R}^{\rm{II}}\oplus
			Q^{\rm{II}}_{7,4,1}\mathfrak{R}^{\rm{II}}\ (k\equiv 1\pmod{8}),\\
			&Q^{\rm{II}}_{6,1,2}\mathfrak{R}^{\rm{II}}\oplus
			Q^{\rm{II}}_{6,2,2}\mathfrak{R}^{\rm{II}}\oplus
			Q^{\rm{II}}_{6,3,2}\mathfrak{R}^{\rm{II}}\oplus
			Q^{\rm{II}}_{6,4,2}\mathfrak{R}^{\rm{II}}\ (k\equiv 2\pmod{8}),\\
			&Q^{\rm{II}}_{5,1,3}\mathfrak{R}^{\rm{II}}\oplus
			Q^{\rm{II}}_{5,2,3}\mathfrak{R}^{\rm{II}}\oplus
			Q^{\rm{II}}_{5,3,3}\mathfrak{R}^{\rm{II}}\oplus
			Q^{\rm{II}}_{13,1,3}\mathfrak{R}^{\rm{II}}\ (k\equiv 3\pmod{8}),\\
			&Q^{\rm{II}}_{4,1,4}\mathfrak{R}^{\rm{II}}\oplus
			Q^{\rm{II}}_{4,2,4}\mathfrak{R}^{\rm{II}}\oplus
			Q^{\rm{II}}_{4,3,4}\mathfrak{R}^{\rm{II}}\oplus
			Q^{\rm{II}}_{12,1,4}\mathfrak{R}^{\rm{II}}\ (k\equiv 4\pmod{8}),\\
			&Q^{\rm{II}}_{3,1,5}\mathfrak{R}^{\rm{II}}\oplus
			Q^{\rm{II}}_{3,2,5}\mathfrak{R}^{\rm{II}}\oplus
			Q^{\rm{II}}_{11,1,5}\mathfrak{R}^{\rm{II}}\oplus
			Q^{\rm{II}}_{11,2,5}\mathfrak{R}^{\rm{II}}\ (k\equiv 5\pmod{8}),\\
			&Q^{\rm{II}}_{2,1,6}\mathfrak{R}^{\rm{II}}\oplus
			Q^{\rm{II}}_{2,2,6}\mathfrak{R}^{\rm{II}}\oplus
			Q^{\rm{II}}_{10,1,6}\mathfrak{R}^{\rm{II}}\oplus
			Q^{\rm{II}}_{10,2,6}\mathfrak{R}^{\rm{II}}\ (k\equiv 6\pmod{8}),\\
			&Q^{\rm{II}}_{1,1,7}\mathfrak{R}^{\rm{II}}\oplus
			Q^{\rm{II}}_{9,1,7}\mathfrak{R}^{\rm{II}}\oplus
			Q^{\rm{II}}_{9,2,7}\mathfrak{R}^{\rm{II}}\oplus
			Q^{\rm{II}}_{9,3,7}\mathfrak{R}^{\rm{II}}\ (k\equiv 7\pmod{8}), 
		\end{cases}
		\]
		where the $P^{\rm{II}}_{\ast,\ast}$ and $Q^{\rm{II}}_{\ast,\ast,\ast}$ are 
		listed in~\cite{miezakiweb}.

		\item 
		Let $\mathfrak{R}^{\rm{III}}=\CC[P^{\rm{III}}_{4,1},P^{\rm{III}}_{4,2}]$. Then we have 
		\[
		I_{G_{m}^{\rm{III}},\chi_d}=
		\begin{cases}
			&\mathfrak{R}^{\rm{III}}\oplus
			Q^{\rm{III}}_{4,1,0}\mathfrak{R}^{\rm{III}}\ (k\equiv 0\pmod{4}),\\
			&Q^{\rm{III}}_{3,1,1}\mathfrak{R}^{\rm{III}}\oplus
			Q^{\rm{III}}_{3,2,1}\mathfrak{R}^{\rm{III}}\ (k\equiv 1\pmod{4}),\\
			&Q^{\rm{III}}_{2,1,2}\mathfrak{R}^{\rm{III}}\oplus
			Q^{\rm{III}}_{2,2,2}\mathfrak{R}^{\rm{III}}\ (k\equiv 2\pmod{4}),\\
			&Q^{\rm{III}}_{1,1,3}\mathfrak{R}^{\rm{III}}\oplus
			Q^{\rm{III}}_{5,1,3}\mathfrak{R}^{\rm{III}}\ (k\equiv 3\pmod{4}),
		\end{cases}
		\]
		where the $P^{\rm{III}}_{\ast,\ast}$ and $Q^{\rm{III}}_{\ast,\ast,\ast}$ are 
		listed in~\cite{miezakiweb}. 
				
		\item
		Let $\mathfrak{R}^{\rm{IV}}=\CC[P^{\rm{IV}}_{2,1},P^{\rm{IV}}_{2,2}]$. 
        Then we have 
		\[
		I_{G_{m}^{\rm{IV}},\chi^{\rm{IV}}_d}=
		\begin{cases}
			&\mathfrak{R}^{\rm{IV}}\ (k\equiv 0\pmod{2}),\\
			&Q^{\rm{IV}}_{1,1,1}\mathfrak{R}^{\rm{IV}}\ (k\equiv 1\pmod{2}),
		\end{cases}
		\]
		where the $P^{\rm{IV}}_{\ast,\ast}$ and $Q^{\rm{IV}}_{\ast,\ast,\ast}$ are 
		listed~\cite{miezakiweb}. 

	\end{enumerate}
	
\end{thm}

\begin{proof}
	We give a proof of $k\equiv0\pmod{8}$ in (2). 
	The other cases can be proved similarly. 
	
	Let 
	\[
	a_{\chi^{\rm{II}}_d,k}=\dim\{P\in I_{G_m^{\rm{II}},\chi_d^{\rm{II}}}\mid \deg(P)=k\}. 
	\]
	
	We can compute $a_{\chi^{\rm{II}}_d,k}$ by Molien's series: 
	\[
	\sum_{d\geq 0}a_{\chi^{\rm{II}}_d,k}t^k
	=\frac{1}{|G_m^{\rm{II}}|}\sum_{g\in G_m^{\rm{II}}}
	\frac{\overline{\chi^{\rm{II}}_d(g)}}{\det(I-tg)}. 
	\]
	(For the details, see~\cite{Molien}, and \cite{Bachoc,CS1999,MS1977}.)
	Then we have 
	\[
	\sum_{d\geq 0}a_{\chi^{\rm{II}}_d,k}t^k=\frac{1 + 3 t^8}{(1 - t^8)^2}. 
	\]
	It is easy to verify that the polynomials 
	$P^{\rm{II}}_{8,1},P^{\rm{II}}_{8,2}, 
	Q^{\rm{II}}_{8,1,0},
	Q^{\rm{II}}_{8,2,0},
	Q^{\rm{II}}_{8,3,0}
	$ belong to the spaces $I_{G_m^{\rm{II}},\chi^{\rm{II}}_d}$ and 
	the result then follows from the equality of the dimensions.
\end{proof}

\section{Concluding remarks}

Crapo and Rota's~\cite{CR1970}
Critical Theorem shows how to count the number of codeword supports of a linear code by evaluating a polynomial called the characteristic polynomial, defined on the contraction of a matroid associated to the code. An extension of the Critical Theorem was given
by Britz and Shiromoto~\cite{BS2008} by presenting a relation between 
the $m$-tuple support weight enumerator of a linear code and the generalized coboundary polynomial of 
a matroid from the code. 
Moreover, Wei~\cite{Wei1991} presented a celebrated duality theorem, 
known as Wei's Duality Theorem, 
that established a remarkable relation between 
the generalized Hamming weights of a linear code and those of its dual code. 
The above discussions give rise to a natural question: 
is there a harmonic generalization of the Critical Theorem as well as the Wei's celebrated Duality Theorem?
We shall investigate a suitable setting in some of our subsequent papers that answers this question.

\section*{Acknowledgements}
\noindent
The authors would also like to thank the anonymous
reviewers for their beneficial comments on an earlier version of the manuscript.
The fourth named author is supported by JSPS KAKENHI (22K03277).

\section*{Data availability statement}
\noindent
The data that support the findings of this study are available from
the corresponding author.

\end{document}